\documentclass[10pt]{amsart}
\usepackage{amsmath, amsthm, amssymb}
\usepackage{mathscinet}
\usepackage{graphicx}
\usepackage{mathbbol}
\usepackage{txfonts}
\usepackage[usenames]{color}

\newtheorem{thm}{Theorem}[section]
\newtheorem{lem}[thm]{Lemma} 
\theoremstyle{definition}
\newtheorem{defn}[thm]{Definition}
\newtheorem{ex}[thm]{Example}

\theoremstyle{remark}
\newtheorem{rmrk}[thm]{Remark}  
   
\numberwithin{equation}{section}

\newcommand{\be}{\begin{equation}}
\newcommand{\ee}{\end{equation}}
\newcommand{\bee}{\begin{equation*}}
\newcommand{\eee}{\end{equation*}}

\newcommand{\R}{\mathbb{R}}

\newcommand{\union}{\cup}

\newcommand{\inner}[1]{\left\langle#1\right\rangle}

\newcommand{\diam}{\operatorname{Diam}}

\newcommand{\vol}{\operatorname{Vol}}
\newcommand{\ric}{\rm{Ric}}
\newcommand{\cut}{\operatorname{cut}}
\newcommand{\hess}{\operatorname{Hess}}

\newcommand{\GHto}{\stackrel{\textrm{GH}}{\longrightarrow}}
\newcommand{\Fto}{\stackrel {\mathcal{F}}{\longrightarrow} }

\newcommand{\Lip}{\operatorname{Lip}}

\newcommand{\intcurr}{{\mathbf I}}      

\newcommand{\spt}{\operatorname{spt}}

\begin{document}

\title{Volumes and Limits of Manifolds with Ricci Curvature and Mean Curvature Bounds}
\author{Raquel Perales}
\thanks{The author's research was funded in part by Prof. Sormani's
NSF Grant DMS 10060059. In addition, the author received funding from Stony Brook as a doctoral student.}
\address{SUNY at Stony Brook}
\email{praquel@math.sunysb.edu}



\begin{abstract}
We consider smooth Riemannian manifolds with nonnegative Ricci curvature 
and smooth boundary.   First we prove a global Laplace comparison theorem 
in the barrier sense for the distance to the boundary.  We apply this theorem 
to obtain volume estimates of the manifold and of regions of the manifold near
the boundary depending upon an upper bound on the area and on the inward 
pointing mean curvature of the boundary.  We prove that families of oriented
manifolds with uniform bounds of this type are compact with respect to the
Sormani-Wenger Intrinsic Flat (SWIF) distance.   
\end{abstract}

\maketitle
\newpage

\section{Introduction}

In the past few decades many important compactness theorems have been 
proven for families of smooth manifolds without boundary.  Gromov has proven that 
families of manifolds with nonnegative Ricci curvature and uniformly
bounded diameter are precompact in the Gromov-Hausdorff (GH) sense \cite{Gromov-metric}.
Cheeger-Colding have proven many beautiful properties of the GH limits of these 
manifolds including rectifiability of the GH limit spaces \cite{ChCo-PartI}.   

Little is known about the precompactness of families of manifolds with boundary. In particular, it is unknown whether sequences of manifolds with
nonnegative Ricci curvature and uniformly bounded mean curvature and area
of the boundary are precompact in the GH sense.  Nor is it known whether the GH
limits of such sequences are rectifiable.

Kodani \cite{Kodani-1990} has proven GH precompactness of
families with uniform bounds on sectional curvature.  Wong  \cite{Wong-2008} has proven GH precompactness of families with uniform bounds for the Ricci curvature, the second fundamental form and the diameter. Neither Kodani nor Wong study the rectifiability
of the GH limit spaces of manifolds in the families they study.  Anderson-Katsuda-Kurylev-Lassas-Taylor \cite{Anderson-2004} and Knox \cite{Knox-2012} have proven $C^{1,\alpha}$ and rectifiability of the limit spaces assuming
one has sequences with significant additional bounds on their manifolds.  See \cite{PerSurvey} for a survey of these precompactness theorems for manifolds with boundary.

We prove precompactness theorems, Theorems 1.4 and 1.5, for families of oriented Riemannian manifolds $(M,g)$ with nonnegative Ricci curvature and uniform upper bounds on the area and the
inward pointing mean curvature of the boundary:
\be
\vol(\partial M)\le A \textrm{ and }  H_{\partial M}(q)\le H. 
\ee
Our precompactness is with respect to the Sormani-Wenger Intrinsic Flat (SWIF)
distance, in which the limit spaces are countably $\mathcal{H}^n$ rectifiable, where $\mathcal{H}^n$ denotes the $n$-dimensional Hausdorff measure.

One important feature of the SWIF distance is that if a sequence of oriented manifolds with volume and area uniformly bounded converges in GH sense then a subsequence converges in SWIF sense. Moreover, if the SWIF limit space is not the zero current space then it can be seen as a subspace of the GH limit. See Theorem 3.20 in \cite{SorWen2}. Nonetheless, SWIF convergence does not imply GH subconvergence as can be seen in Example \ref{ex-IlmanemSeq}.   In \cite{SorWen1} Sormani-Wenger have shown that for manifolds with non negative Ricci curvature (without boundary) the GH and SWIF limits agree. This is not necessarily true for manifolds with boundary. For example, consider a sequence of $n$-closed round balls of the same radii with one increasingly thin tip. The sequence converges in SWIF sense to a round   closed ball but it converges to a round closed ball with a segment attached in the GH sense (c.f. Example A.4 in \cite{SorWen2}).

In order to prove our precompactness theorem we need to prove theorems for manifolds with boundary that were previously proven for manifolds with no boundary. One of the key tools in the work of Gromov is the Bishop-Gromov Volume Comparison Theorem \cite{Gromov-metric}.   A key tool in the work of Cheeger-Colding is the Abresch-Gromoll Laplace Comparison Theorem \cite{Abresch-Gromoll}.  In fact, the Abresch-Gromoll Laplace Comparison Theorem may be applied to prove the Bishop-Gromov Volume Comparison Theorem (c.f. \cite{Cheeger-Criticalp}).    

In Section 2 we prove Theorems \ref{thm-Blaplaciancomparison}, \ref{thm-volumebounds} and  \ref{thm-areabounds}. We consider connected Riemannian manifolds $(M^n, g)$ with smooth boundary $\partial M$.
We denote by $d:M \times M \to \R$ the metric on $(M,g)$ given by $g$. Suppose that 
$(M,d)$ is a complete metric space. Define $r: M \to \R$ by 
\be 
r(p):=d(p, \partial M).
\ee
The laplacian of $r$ is denoted by $\Delta r$. The mean curvature of $\partial M$ with respect to the the normal inward pointing direction is denoted by $H_{\partial M}: \partial M \to \R$. 

\begin{thm}\label{thm-Blaplaciancomparison}
Let $n \geq 2$ and $M^n$  be an $n$-dimensional connected Riemannian manifold with boundary with $\ric(M \setminus \partial M) \geq 0$ and $(M,d)$ complete. Then for all $p \in M$ 
\be \label{eq-laplacian}
\Delta r (p)\leq \frac{(n-1)H_{\partial M}(q) }{H_{\partial M}(q) r(p)+n-1}
\ee
holds in the barrier sense, where $q \in \partial M$ such that $r(p)=d(p,q)$.
\end{thm}

In Theorem \ref{thm-Blaplaciancomparison} we get $\Delta r(p) \leq 0$ when $H_{\partial M}=0$. If $H_{\partial M}= (n-1)/H$ then $\Delta r(p) \leq (n-1)/(r(p)+H)$.

Sakurai has recently proven a Laplacian comparison theorem for the same distance function
whenever $r$ is smooth \cite{Sakurai} \footnote{His paper appeared on the
arxiv after our original posting.}.  In our paper we also include points where $r$ is not smooth obtaining a global Laplacian comparison theorem in the barrier sense. Abresch-Gromoll's original Laplacian comparison theorem is also proven globally in the barrier sense for distance functions on manifolds without boundary \cite{Abresch-Gromoll}. This
global comparison allows one to apply the maximum principle and has much
stronger consequences than a Laplacian comparison theorem which only holds
where the function is smooth.

In Subsection \ref{subsec-VolBounds} we apply Theorem \ref{thm-Blaplaciancomparison} to obtain volume and area estimates for $M^{\delta_2}\setminus M^{\delta_1}$ and $\partial M^\delta$, respectively, where
\be 
M^{\delta}: = \{ p \in M | r(p)> \delta \}
\ee
and $\partial M^\delta$ is the boundary (as a metric subspace of $M$) of $M^\delta$. Note that $\partial M^\delta \subset r^{-1}(\delta)$ but these sets are not necessarily equal.
Different volume estimates were obtained by Heintz and Karcher in \cite{VolComparison} using Jacobi fields. In our theorem, $A_{n,H}: [0, \infty) \to \R$ is the function given by
\begin{align}\label{eq-AHn}
A_{n,H}(\delta) =
\begin{cases}
(H\delta+n-1)^{n-1}/(n-1)^{n-1}& \textrm{ if } H\delta+n-1\geq 0,\\
0 &\textrm{ otherwise, } 
\end{cases}
\end{align}
where $n \geq 2$ and $H \in \R$.

\begin{thm}\label{thm-volumebounds}
Let $n \geq 2$ and $M^n$ be an $n$-dimensional and connected Riemannian manifold with smooth boundary such that $(M,d)$ is complete, $\ric(M\setminus \partial M) \geq 0$, $H_{\partial M} \leq H$ and $\vol(\partial M) \leq \infty$. 
If $\delta_1 \geq \delta_2 \geq 0$ then
\be 
\vol(M^{\delta_2} \setminus M^{\delta_1}) \leq \vol(\partial M) \int_{\delta_2}^{\delta_1}
A_{n,H}(t)dt,
\ee
where $A_{n,H}$ is as in (\ref{eq-AHn}).
If $\diam(M) \leq D$ then 
\be 
\vol(M)\leq \vol(\partial M) \int_{0}^{\tilde{D}} A_{n,H}(t)dt, 
\ee
where $\tilde{D}=D$ if $H\geq0$ and $\tilde{D}=\min\{D, -(n-1)/H\}$
if $H<0$.
\end{thm}

Explicitly, the integral of $A_{n,H}$ is the following 

\begin{align}
\int_{\delta_2}^{\delta_1}A_{n,H}(t)dt=
\begin{cases}
\delta_1- \delta_2 &\textrm{ if } H=0\\
\frac{n-1}{nH}
\left(\frac{H\delta+ n-1}{n-1}\right)^n \bigg
\rvert_{\delta_2}^{\tilde{\delta_1}} 
&  \textrm{ if }  H \neq 0,
\end{cases}
\end{align}
where $\tilde{\delta_1}=\delta_1$ if $H\geq0$ and $\tilde{\delta_1}=\min\{\delta_1, -(n-1)/H\}$
if $H<0$.

We see that the equality of both, volume and area, estimates is achieved by all the Riemannian manifolds of the sequence given in Example \ref{ex-diam}. Also for the standard ball of radius $R$ in $n$-euclidean space.

\begin{thm}\label{thm-areabounds}
Let $n \geq 2$, $M^n$  be an $n$-dimensional connected Riemannian manifold with smooth boundary such that $(M,d)$ is complete, $\ric(M\setminus \partial M) \geq 0$, $H_{\partial M} \leq H$ and $\vol(\partial M) \leq \infty$. Then, 
$\mathcal{L}^1$-almost everywhere, 
\be 
\vol(  \partial M^\delta ) \leq \vol(\partial M) A_{n,H}(\delta),
\ee
where $A_{n,H}$ is as in (\ref{eq-AHn}).
\end{thm}

In Section \ref{sec-IFC} we review some basic definitions about SWIF distance such as Wenger's compactness theorem. Wenger showed that given a sequence of complete oriented Riemannian manifolds of the same dimension with $\,\vol(M_j)\le V\,$, $\,\vol(\partial M_j) \le A\,$ and $\diam(M_j) \le D$ a subsequence converges in the SWIF sense to an integral current space. See Theorem 1.2 in \cite{Wenger-compactness}, cf. Theorem 4.9 in \cite{SorWen2}.  We use Wenger's compactness theorem along with the area and volume estimates to prove convergence theorems, Theorems \ref{thm-IFcompact1}, \ref{thm-IFcompact2} and \ref{thm-IFlimits}.

\begin{thm}\label{thm-IFcompact1}
Let $D, A > 0$, $H \in \R$ and $(M^n_j,g_j)$ be a sequence of 
$n$-dimensional oriented connected Riemannian manifolds with smooth boundary. Suppose that for all $j$ the spaces $(M_j,d_j)$ are complete metric spaces that satisfy
\be 
\ric(M_j \setminus \partial M_j)\geq 0,
\ee
\be  
\vol(\partial M_j)\leq A,\,\,\, H_{\partial M_j} \leq H,
\ee
and
\be
\diam(M_j)\leq D. 
\ee
Then there is an $n$-integral current space $(W,d,T)$ and a 
subsequence $(M_{j_k}, d_{j_k},T_{j_k})$ that converges in SWIF sense
\be 
(M_{j_k},d_{j_k},T_{j_k}) \Fto (W,d,T),
\ee
where $T_j(\omega) : = \int_{M_j}\omega$.  
\end{thm} 

The necessity of diameter and mean curvature uniform bounds in Theorem \ref{thm-IFcompact1} can be seen in Example \ref{ex-diam} and Example \ref{ex-jfold}, respectively. 

Myers proved that for a complete Riemannian manifold with Ricci curvature bounded from below geodesics past certain distance must have conjugate points. Thus, the diameter of the manifold is bounded below by this distance. Li and Li-Nguyen in \cite{LiMartin} and \cite{Li-Nguyen}, respectively, proved that if $(M^n, g)$ is a complete connected Riemannian manifold with smooth boundary such that $\ric(M \setminus \partial M) \geq 0$ and $H_{\partial M} \leq H < 0$  then $r \leq - (n-1)/H$. Hence, $\diam(M)$ can be bounded in terms of $-(n-1)/H$ and $\diam(\partial M)$. See Remark \ref{rmrk-boundedDiam}. We get the following compactness theorem.

\begin{thm}\label{thm-IFcompact2}
Let $D', A> 0$ and $(M^n_j,g_j)$ be a sequence of 
n-dimensional oriented connected Riemannian manifolds with smooth boundary. Suppose that for all $j$ the spaces $(M_j,d_j)$ are complete metric spaces that satisfy
\be 
\ric(M_j \setminus \partial M_j)\geq 0,
\ee
\be  
\vol(\partial M_j)\leq A,\,\,\, H_{\partial M_j} \leq H <0,
\ee
and
\be
\diam(\partial M_j)\leq D'. 
\ee
Then there is a subsequence $(M_{j_k}, d_{j_k},T_{j_k})$ and an
$n$-integral current space $(W,d,T)$ such that
\be 
(M_{j_k}, d_{j_k},T_{j_k}) \Fto (W,d,T).
\ee
\end{thm} 

In Example \ref{ex-diam} we describe a sequence that satisfies all the hypotheses 
of Theorem \ref{thm-IFcompact2}, except that $H_{\partial M_j}=0$ for all $j$. This sequence does not converge in SWIF sense. Hence a uniform negative bound on the mean curvature is needed.  

When using GH distance the following can occur. See Example 4.10 of \cite{Persor-2013} and Example \ref{ex-IlmanemSeq}. There exists a sequence of oriented connected Riemannian manifolds with smooth boundary that satisfy $(M_j,d_j)$ is complete as metric space, $\ric(M_j \setminus \partial M_j)\geq 0$, $\vol(\partial M_j)\leq A,\,\,\, H_{\partial M_j} \leq H$ and $\diam(M_j)\leq D$ such that
\be 
(M_j, d_j) \GHto (X,d_X)
\ee 
and for every decreasing sequence $\delta_i \to 0$ 
\be 
(M_{j_k}^{\delta_i}, d_{j_k}) \GHto (Y_{\delta_i},d_{Y_{\delta_i}}),
\ee
but $(Y_{\delta_i}, d_{Y_{\delta_i}})$ does not converge in GH sense to $(X,d_X)$.
In the theorem below we see that this situation does not happen if we replace GH distace by SWIF distance.

\begin{thm}\label{thm-IFlimits}
Let $D, A > 0$, $H \in \R$ and $(M^n_j,g_j)$ be a sequence of 
n-dimensional complete oriented connected Riemannian manifolds with smooth boundary.  Suppose that for all $j$ the spaces $(M_j,d_j)$ are complete metric spaces that satisfy
\be 
\ric(M_j \setminus \partial M_j)\geq 0,
\ee
\be  
\vol(\partial M_j)\leq A,\,\,\, H_{\partial M_j} \leq H,
\ee
and
\be
\diam(M_j)\leq D. 
\ee
Suppose that there exist an integral current space $(W,d,T)$, 
a non increasing sequence $\delta_i \to 0$ and integral current spaces
$(W_{\delta_i},d_{W_{\delta_i}},T_{\delta_i})$ such that 

\be 
(M_j,d_j,T_j) \Fto (W,d,T)
\ee
and for all $i$
\be
(M_{j_k}^{\delta_i}, d_{j_k},T_{j_k}^{\delta_i}) \Fto (W_{\delta_i},d_{W_{\delta_i}},T_{\delta_i}).
\ee
Then we have
\be 
(W_{\delta_i},d_{W_{\delta_i}},T_{\delta_i}) \Fto (W,d,T). 
\ee
\end{thm} 

In Subsection \ref{ssec-examples} we provide examples of sequences of manifolds with boundary. Example \ref{ex-IlmanemSeq} defines a sequence (as in \cite{SorWen2}) that  converges in SWIF sense but not in GH sense.  Example \ref{ex-jfold} defines a sequence  (as in \cite{SorAA}) that shows the necessity of a uniform bound of the mean curvature required in Theorem \ref{thm-IFcompact1}. Example \ref{ex-diam} shows the necessity of a uniform bound of the diameter required in Theorem \ref{thm-IFcompact1} and that equality holds in both volume and area estimates given in Theorem \ref{thm-volumebounds} and Theorem \ref{thm-areabounds}. Moreover, this example shows that the assumption $H<0$ in Theorem \ref{thm-IFcompact2} is needed.

I would like to thank my doctoral advisor, Prof Sormani, for introducing me to 
the notion of intrinsic flat convergence and helping with expository aspects of this
paper. I would like to thank my fellow participants in the 
CUNY Metric Geometry Reading Seminar: Jacobus Portegies, Sajjad Lakzian and Kenneth Knox for their amazing presentations.  I would like to thank Professors Anderson, Khuri, Lawson and LeBrun for their excellent courses and their support. I would like to thank Panagiotis Gianniotis and Pedro Solorzano with whom I have discussed part of this work.   I would like to thank Professors Wilkins, Plaut and Searle for providing me with the opportunity to present this work in Tennessee.

\section{Volume, Area and Diameter Bounds} \label{sect-VolAreaDiam}

In this section we see that the function $r$ is differentiable almost everywhere by showing that it is a Lipschtiz map and invoking Rademacher's theorem. We also give a proof that shows that $r$ is bounded when $H_{\partial M} \leq H <0$. This result is used  to bound the diameter of $M$ in terms of the diameter of $\partial M$ and $H$. We also prove Theorems \ref{thm-Blaplaciancomparison}, \ref{thm-volumebounds} and \ref{thm-areabounds}. 

\begin{lem}\label{thm-distancefuntion}
Let $M^n$ be a connected Riemannian manifold with boundary such that $(M,d)$ is complete as metric space. Then $r=d(\partial M, \,\,)$ is a Lipschitz function with $\Lip(r)=1$. 
\end{lem}

\begin{proof}
Let $p,q\in M$. There exists $p'\in \partial M$ such that
$r(p)=d(p',p)$. Then, 
\be
r(q)-r(p) \leq d(p', q)-d(p',p) \leq d(q,p).
\ee
In the same way it is proven that $r(p)-r(q) \leq d(q,p)$.
Thus, $r$ is a Lipschitz function with $\Lip(r)=1$. 
\end{proof}

The composition of  $r$ with a normal coordinate on a strongly convex ball
is a Lipschitz function. Hence, $r$ is differentiable except for a zero measure set. 

\subsection{Diameter Bounds for Manifolds with Negative Mean Curvature}\label{subsec-DiamBounds}

We give the definitions of a focal point and a cut point of $\partial M$.
Then define the function $\pi: M\setminus \cut(\partial M) \to \partial M$ which assigns 
to each point $p$ in the domain the unique point in the boundary that equals $r(p)$.
Then we prove the theorem of Li and Li-Nguyen, \cite{LiMartin} and \cite{Li-Nguyen}, respectively, that gives an upper bound on $r$ when $H_{\partial M} \leq H < 0$. 
With that bound we get an upper estimate of the diameter of $M$ in Remark \ref{rmrk-boundedDiam}. 

\begin{defn}
$q \in M$ is a focal point of $\partial M$ if there exists a geodesic
$\gamma:[0,a] \to M$ such that 
$\gamma(0) \in \partial M$, 
$\gamma'(0) \in T_{\gamma(0)} \partial M^\perp$ and 
$\gamma(a)=q$, and a Jacobi field $J$ along $\gamma$ that vanishes at $b$ and satisfies $J(0) \in T_{\gamma(0)}\partial M$ and $J'(0)+S_{\gamma'(0)}(j(0)) \in  T_{\gamma(0)}\partial M^\perp$.

A cut point of $\partial M$ is either a first focal point or a point with two geodesics back to the boundary of the same length achieving the
distance to the boundary. Denote by $\cut(\partial M)$ the set of cut points of $\partial M$.
\end{defn}

\begin{rmrk}
If $(M,g)$ is a connected Riemannian manifold with boundary with $(M,d)$ complete as a metric space, then geodecis normal to the boundary are not minimizing past a focal point.
See Section 11.4, Corollary 1 of Theorem 5 in \cite{B-C}.
\end{rmrk}

As a consequence the following function is well defined.

\begin{defn}
Let $\pi: M\setminus \cut(\partial M) \to \partial M$ be the function 
that assigns to $p \in M\setminus \cut(\partial M)$ the only point $\pi(p) \in \partial M$ that satisfies $r(q)=d(\pi(q),q)$.
\end{defn}

\begin{lem}\label{lem-focalpoints}[Li, Li-Nguyen]
Let $M^n$ be an $n$-dimensional and connected Riemannian manifold with boundary such that $\ric(M \setminus \partial M) \geq 0$  and $(M,d)$ is a complete metric space. Suppose that $p' \in \partial M$ has $H_{\partial M}(p') < 0$, then the geodesic that starts at $p'$ with initial vector the unitary normal inward vector stops minimizing after time 
$t_0 > -\frac{n-1}{H_{\partial M}(p')}$.
\end{lem} 

\begin{proof}
Let $E_i$ be an orthonormal basis of parallel fields along $\gamma$ such that $E_n=\gamma'$. Let $V_i(t)=(t_0-t)E_i(t)$ be vector fields along $\gamma$. Then,
\begin{align} 
I_{t_0}(V_i) & =  \int_{0}^{t_0} \lbrace \inner{V_i',V_i'}-\inner{R(\gamma',V_i)\gamma',V_i} \rbrace(t)dt
+\inner{S_{\gamma'}V_i,V_i}(0)-\inner{S_{\gamma'}V_i,V_i}(t_0)\\
& = \int_{0}^{t_0} \lbrace 1 -(t_0-t)^2\inner{R(\gamma',E_i)\gamma',E_i} \rbrace(t)dt
+t_0^2\inner{S_{\gamma'}E_i,E_i}(0).
\end{align}

Now we add $I(V_i)$ and use the fact that $\ric(M \setminus \partial M) \geq 0$
\begin{align} 
\sum_{i=1}^{n-1} I_{t_0}(V_i) & = 
 \int_{0}^{t_0} \lbrace (n-1) -(t_0-t)^2\sum_{i=1}^{n-1} \inner{R(\gamma',E_i)\gamma',E_i} \rbrace(t)dt
+t_0^2\sum_{i=1}^{n-1} \inner{S_{\gamma'}E_i,E_i}(0)\\
& \leq t_0(n-1)+t_0^2H_{\partial M}(p')\\
& = t_0((n-1)+t_0H_{\partial M}(p')).
\end{align}
Assuming that $t_0 > -\frac{n-1}{H_{\partial M}(p')}$ and since $H_{\partial M}(p') < 0$, we get
$(n-1)+H_{\partial M}(p')t_0 < 0$. Then $\sum_{i=1}^{n-1} I_{t_0}(V_i) < 0$.
Thus, there is  $i$ for which $I_{t_0}(V_i) < 0$. Hence, $\gamma$ is not minimizing.
\end{proof}

\begin{rmrk}\label{rmrk-boundedDiam}
The lemma implies that if $H_{\partial M} \leq H <0$  and $\diam(\partial M)\leq D'$ then $\diam(M)$ is bounded. For $p,q \in M$
\begin{align} 
d_M(p,q) & \leq d_M(p,\pi(p))+d_M(\pi(p),\pi(q))+d_M(\pi(q),q)\\
& \leq -\frac{n-1}{H} + d_{\partial M}(\pi(p),\pi(q)) -\frac{n-1}{H} \\
& = D' -2 \frac{n-1}{H},
\end{align}
where we use that the intrinsic metric $d_{\partial M}$ on $\partial M$ is greater or equal than the 
restricted metric $d_M|_{\partial M}$.
\end{rmrk}

\subsection{Laplacian Comparison Theorems}\label{subsec-Laplacian}

For manifolds with no boundary two Laplacian comparison theorems for the function distance to a point were proven. The first was proven only for points outside the cut locus of the point. Then it was extended to the barrier sense. See \cite{Petersen-text} and \cite{B-C}. We also prove a Laplacian comparison theorem for $r$ for points outside $\cut(\partial M)$, Theorem \ref{thm-Blaplaciancomparison}. We define upper barrier function and laplacian comparison in the barrier sense (see \cite{Cheeger-Criticalp}). Then we prove a Laplacian comparison theorem in the barrier sense, Theorem \ref{thm-Blaplaciancomparison}.

\begin{thm}\label{thm-laplaciancomparison}
Let $M^n$  be an $n$-dimensional connected Riemannian manifold with boundary with $(M,d)$ complete as metric space, $\ric(M \setminus \partial M) \geq 0$.
Then for all $p \in M  \setminus \cut (\partial M) $
\be 
\Delta r(p) \leq \frac{(n-1)H_{\partial M}(\pi(p))}{H_{\partial M}(\pi(p))r(p)+n-1}
\ee
\end{thm}

\begin{proof}
For points in $\partial M$ the result is true by hypothesis.
For points in $M \setminus (\cut (\partial M) \cup \partial M)$ we use
Bochner-Weitzenbock's formula 
\be 
|\hess r|^2 + \tfrac{\partial}{\partial r}(\Delta r)+\ric(\nabla r, \nabla r)=0.
\ee
Since $|\hess r|^2 \geq  \tfrac{(\Delta r)^2}{n-1}$ and $\ric(M) \geq 0$,
\be \label{eq-bochner}  
0 \geq \tfrac{(\Delta r)^2}{n-1}
+ \tfrac{\partial }{\partial r}(\Delta r).
\ee

Take $p \in M  \setminus (\cut (\partial M) \cup \partial M)$. Let $\gamma$ be the minimizing geodesic from $\pi(p)$ to $p$. First we assume that $\Delta r(\gamma (t)) \neq 0$ later we prove that it was not necessary. We arrange terms in (\ref{eq-bochner}) and integrate along $\gamma$

\be  
-\int_{0}^{r(p)} \frac{ \tfrac{\partial}{\partial r}(\Delta r)}{(\Delta r)^2} \geq \int_{0}^{r(p)}\frac{1}{n-1},
\ee
get
\begin{align}
\frac{1}{\Delta r(p)} & \geq \frac{r(p)}{n-1} + \frac{1}{\Delta r(\pi(p))}\\
& = \frac{r(p)}{n-1} + \frac{1}{H_{\partial M}(\pi(p))} \\
&  = \frac{H_{\partial M}(\pi(p))r(p)+n-1}{(n-1)H_{\partial M}(\pi(p))}.
\end{align}  

If $H_{\partial M}(\pi(p)) <0$ then by Lemma \ref{lem-focalpoints}
both sides of the inequality are negative. If $H_{\partial M}(\pi(p)) >0$
then both sides of the inequality are positive. Thus 
\be \label{eq-smoothcomparisonf}
\Delta r(p) \leq \frac{(n-1)H_{\partial M}(\pi(p)) }{H_{\partial M}(\pi(p)) r(p)+n-1}.
\ee

Finally, we deal with the case in which there is $t_0 \geq 0$ such that $ \Delta r(\gamma (t_0)) = 0$. Suppose that $t_0=\inf\{t |\Delta r(\gamma (t))=0\}$. By (\ref{eq-bochner}),  $\Delta r(\gamma (t))$ is non increasing. Since $0 \leq t_0$ and $ \Delta r(\gamma (t_0)) = 0$, then $H_{\partial M}(\pi(p))=\Delta r(\gamma (0)) \geq 0$. This means that the right hand side of (\ref{eq-smoothcomparisonf}) is nonnegative for $t \in [0,t_0]$. Using again that $\Delta r(\gamma (t))$ is non increasing we see that $\Delta r(\gamma (t)) \leq 0$ for $t \geq t_0$. Thus, (\ref{eq-smoothcomparisonf}) holds for $t \geq t_0$.
\end{proof}

\begin{defn}
Let $f$ be a continuous real valued function. An upper barrier for $f$ at the point $x_0$
is a $C^2$ function $f_{x_0}$ defined in some neighborhood of $x_0$ such that
$f \leq f_{x_0}$ and $f(x_0)=f_{x_0}(x_0)$.
\end{defn}

\begin{defn}
Let $f$ be a continuous function. $\Delta f(x_0) \leq a$ in the barrier sense
if for all $\varepsilon >0$ there is an upper barrier $f_{x_0,\varepsilon}: U_\varepsilon \to \R$ for $f$ at $x_0$ with 
\be 
\Delta f_{x_0,\varepsilon} \leq a+\varepsilon.
\ee
\end{defn}

Now we are ready to extend Theorem \ref{thm-laplaciancomparison}.

\begin{proof}[Proof of Theorem \ref{thm-Blaplaciancomparison}]
For points in $M \setminus \cut(\partial M)$
the result follows by applying Theorem \ref{thm-laplaciancomparison}.
Suppose that $p \in \cut(\partial M)$. Take $q \in \partial M$ such that $r(p)=d(p,q)$. If $H_{\partial M}(q) r(p)+n-1=0$ there is nothing to prove. Otherwise, for all $\varepsilon >0$  we will define an upper barrier $r_{p,\varepsilon}: U_\varepsilon \to \R$ for $r$ at $p$ such that in $U_\varepsilon$
\be \label{eq-laplacianbarrier}
\Delta r_{p,\varepsilon} \leq \frac{(n-1)H_{\partial M}(q) }{H_{\partial M}(q) r(p)+n-1} +\varepsilon.
\ee

Let $U$ be an open set of $\partial M$ that contains $q$ such that the map $U \times [0,\delta_0) \to M$ given by $(z,t) \mapsto \exp(t\nabla r(z))$ is a diffeormophism. Let $x=(x_1,...,x_{n-1}):B(q) \to \R^{n-1}$ be a coordinate chart centered at $q$ such that $\partial_i (q)=\frac{\partial}{\partial x_i}(q)$ are orthonormal and $\bar{B}(q) \subset U$ is a closed ball centered at $q$.

We construct upper barrier functions by constructing distance functions to $(n-1)$-submanifolds $N_{\delta,\alpha} \subset M$. 
For $\delta \leq \delta_0$ and 
$\alpha > 0$ let 
$g: U \to \R$ be a smooth function that satisfies 
$g(q)=\delta$, at other points 
$0 \leq g < \delta$, $g$ has a maximum at $q$ and $\sum g_i''(0) \geq -\alpha$,
where $g_i$ is the i-th coordinate function of $g \circ x^{-1}$.  By the existence of partitions of unity, there is a smooth function $\varphi: \partial M \to \R$ that satisfies $0 \leq \varphi \leq 1$, $\varphi = 1$ in $\bar{B}(q)$ and $\spt \varphi \subset U$.
Hence, we can suppose that $g: \partial M \to \R$ is a smooth function such that
$g(q)=\delta$, at other points $0 \leq g < \delta$, $g$ has a maximum at $q$, $\sum g_i''(0) \geq -\alpha$ and $\spt g \subset U$.

We define
\be  
N_{\delta,\alpha}:= \left \{ \exp(g(z) \nabla r (z)) : z \in \partial M \right \}.
\ee

We claim that $\delta$ and $\alpha$ can be chosen such that there is a neighborhood $U_\varepsilon$ of $p$ for which
$r_{p,\varepsilon}: U_\varepsilon \to \R$ given by 
\be \label{eq-upperbf}
r_{p,\varepsilon}=d(\,\cdot\,, N_{\delta,\alpha}) + \delta
\ee is an upper barrier of $r$ at $p$ that satisfies (\ref{eq-laplacianbarrier}).

For $y \in M$
\be \label{eq-uppersatisfied}
r_{p,\varepsilon}(y)
=d(y, N_{\delta,\alpha}) + \delta
\leq \inf_{z \in  \partial M}\{r(y) - g(z)\} + \delta 
=r(y) + \inf_{z \in  \partial M}\{\delta - g(z)\}.
\ee
By definition of $g$, $\delta - g$ is nonnegative and it is zero only when $z=q$. Thus 
$r_{p,\varepsilon}(y) \leq r(y)$. Also $z=q$ is the only point in $\partial M$
for which $r(p) + (\delta - g(z))$ equals $r(p)$. 

$r_{p,\varepsilon}(p) = r(p)$ and $p$ is not in the cut locus of $N_{\delta,\alpha}$. Hence, there is a neighborhood of $p$ in which $r_{p,\varepsilon}$ is $C^2$. 

It remains to prove that (\ref{eq-laplacianbarrier}) is true. 
This follows by continuity of $\Delta r_{p,\varepsilon}$ at $p$,
Theorem \ref{thm-laplaciancomparison} applied to the function $r_{p,\varepsilon}$ and continuity of the functions
\be\label{eq-comparisonf}   
(\delta, \alpha) \mapsto \frac{(n-1)H_{N_{\delta,\alpha}}(\exp(g(q)\nabla r(q))) }{H_{N_{\delta,\alpha}}(\exp(g(q)\nabla r(q))) r_{p,\varepsilon}(p)+n-1}
\ee
and
\be \label{eq-comparisonf1var}
\tilde{q} \mapsto \frac{(n-1)\Delta r (\tilde{q})}{\Delta r (\tilde{q}) r_{p,\varepsilon}(p)+n-1}
\ee
at $(0,0)$ and $q$, respectively, where $H_{N_{\delta,\alpha}}$ denotes
the mean curvature of $N_{\delta,\alpha}$ in the inward normal direction. More explicitly, for all $y$ in a neighborhood $U_\varepsilon$ of $p$ the following is satisfied
\begin{align} 
\Delta r_{p,\varepsilon}(y) & < \tfrac{\varepsilon}{3} + \Delta r_{p,\varepsilon}(p)   \\
& < \tfrac{\varepsilon}{3} +
 \frac{(n-1)H_{N_{\delta,\alpha}}(\exp(g(q)\nabla r(q))) }{H_{N_{\delta,\alpha}}(\exp(g(q)\nabla r(q))) r_{p,\varepsilon}(p)+n-1} \\
& < \tfrac{\varepsilon}{3}+ \tfrac{\varepsilon}{3} +
 \frac{(n-1) \Delta r (\exp(g(q)\nabla r(q))) }{\Delta r (\exp(g(q)\nabla r(q))) r(p)+n-1}\\
& < \tfrac{\varepsilon}{3}+ \tfrac{\varepsilon}{3}+ \tfrac{\varepsilon}{3} + \frac{(n-1)H_{\partial M}(q) }{H_{\partial M}(q) r(p)+n-1}.
\end{align}

Since $H_{\partial M}(q) r(p)+(n-1) \neq 0$ (\ref{eq-comparisonf1var}) is continuous at $q$. The continuity at $(0,0)$ of the function given in (\ref{eq-comparisonf}) follows from the continuity at $(0,0)$ of
$(\delta, \alpha) \mapsto H_{N_{\delta,\alpha}}(\exp(g(q)\nabla r(q)))$.

Let's calculate $H_{N_{\delta,\alpha}}(\tilde{q})$ where $\tilde{q}=\exp(g(q)\nabla r(q))$. Recall that the map $U \times [0,\delta_0) \to M$ given by $(z,t) \mapsto \exp(t\nabla r(z))$ is a diffeormophism and that $(x_1,...,x_{n-1}):B(q) \subset U \to \R^{n-1}$ is a coordinate system centered at $q$ such that $\partial_i (q)=\frac{\partial}{\partial x_i}(q)$ are orthonormal. Thus, we can suppose that $\partial_i $ are vector fields defined on $U \times [0,\delta_0)$.
Then, the tangent space of $N_{\delta,\alpha}$ at $\tilde{q}$
is spanned by $E_i (\tilde{q}): = \partial _i (\tilde{q})+ g_i'(0)\partial_n (\tilde{q})$, $i=1,...,n-1$, where $\partial_n  = \nabla r $. 

\be 
\nabla_{E_i} E_i (\tilde{q}) 
= \nabla_{\partial_i} (\partial _i + g_i'(0)  \partial_n) (\tilde{q}) 
=( \nabla_{\partial_i} \partial _i + 
g_i'(0) \nabla_{\partial_i} \partial_n +
g_i''(0) \partial_n )(\tilde{q}) = ( \nabla_{\partial_i} \partial _i +
g_i''(0) \partial_n )(\tilde{q})
\ee

\begin{align} 
H_{N_{\delta,\alpha}}(\tilde{q}) & = - \sum_{i=1}^{n-1} \inner{ \nabla r_{p,\varepsilon},\nabla_{E_i} E_i}(\tilde{q}) \\
& = - \sum_{i=1}^{n-1} \inner{ \partial_n,\nabla_{\partial_i} \partial _i +
g_i''(0)  \partial_n}(\tilde{q}) \\
& = \Delta r (\tilde{q}) - \sum g_i''(0) \leq \Delta r (\tilde{q}) +\alpha.
\end{align}
\end{proof}

\subsection{Volume and Area Estimates}\label{subsec-VolBounds}

Recall that for $\delta>0$,  
\be 
M^{\delta}= \{ p \in M | r(p)> \delta \}
\ee
and $\partial M^\delta$ is the boundary (as a metric subspace of $M$) of $M^\delta$. In this subsection, area and volume estimates of $\partial M^\delta$,  and annular regions, $M^{\delta_2} \setminus M^{\delta_1}$, respectively, are proven. 

Using the normal exponential map we can write the volume form of $M$ at a point $p=\exp_x(t\nabla r (x))$ as $A(x,t)dm(x)dt$, where $x \in \partial M$. In the following lemma we bound $A(x,t)$.

\begin{lem}\label{lem-SmoothArea}
Let $M^n$  be an $n$-dimensional Riemannian manifold with smooth boundary such that $(M,d)$ is complete as metric space, $\ric(M\setminus \partial M) \geq 0$ and $H_{\partial M} \leq H$. In $M\setminus \cut(\partial M)$ write
the volume form of $M$ as $A(x,t)dm(x)dt$, where $dm(x)$ is the volume form of $\partial M$. 
Then,
\begin{align}
A(x,\delta) \leq  A(x,0) A_{n,H}(\delta).
\end{align}
\end{lem}

Note that when $H=0$,  $A_{n,H}=1$. Thus, $A(x,\delta) \leq A(x,0)$.

\begin{proof}
Let $p \in M\setminus (\cut(\partial M)\union \partial M)$. Then there is $(x,\delta) \in \partial M \times \R$ such that $r(p)=\delta=d(x,p)$. Let $\gamma$ be the minimizing geodesic from $x$ to $p$. Note that if $H_{\partial M}(x) \leq H$ then 
\be 
\frac{(n-1)H_{\partial M}(x) }{H_{\partial M}(x) r(\gamma(t))+n-1} \leq 
\frac{(n-1)H }{H t+n-1}.
\ee 
Thus, by Theorem \ref{thm-laplaciancomparison} and since $\Delta r= A'/A$ 
\begin{align}\label{eq-areafunction}
\frac{A'}{A}(x,t) \leq \frac{(n-1)H}{Ht +n-1}.
\end{align}
By Lemma \ref{lem-focalpoints}, $Ht +n-1 \neq 0$ for $0 \leq t \leq \delta$. Thus, integrating
(\ref{eq-areafunction}) with respect to $t$ from $0$ to $\delta$ we get
\begin{align}
\ln \left( \frac{A(x,\delta)}{A(x,0)}\right) 
\leq (n-1) \ln \left(\frac{H\delta+ n-1}{n-1}\right).
\end{align}
Taking exponentials in both sides of the inequality and arranging terms:
\begin{align}
A(x,\delta)
\leq A(x,0) \left(\frac{H\delta+ n-1}{n-1}\right)^{n-1}= A(x,0) A_{n,H}(\delta).
\end{align}
\end{proof}

Using this estimate we obtain bounds for the volume of annular regions, $M^{\delta_2} \setminus M^{\delta_1}$.

\begin{proof}[Proof of Theorem \ref{thm-volumebounds}]

\begin{align}
\vol(M^{\delta_2} \setminus M^{\delta_1}) & =
\int_{\delta_2}^{\delta_1}
\int_{x \in \partial M} 
A(x,t)dm(x)dt \\
& \leq 
\int_{\delta_2}^{\delta_1}
\int_{x \in \partial M} 
A_{n,H}(t)dm(x)dt \\
& = \vol(\partial M) \int_{\delta_2}^{\delta_1}A_{n,H}(t)dt
\end{align}

For $H=0$ we have that
\be
\int_{\delta_2}^{\delta_1}A_{n,H}(t)dt=\delta_1-\delta_2. 
\ee
For $H \neq 0$ we get
\begin{align}
\int_{\delta_2}^{\delta_1}A_{n,H}(t)dt
& = \frac{n-1}{Hn}\vol(\partial M)
\left(\frac{H\delta+ n-1}{n-1}\right)^n \bigg
\rvert_{\tilde{\delta_2}}^{\tilde{\delta_1}},
\end{align}
where $\tilde{\delta_i}=\delta_i$ if $H>0$. If $H<0$, by definition $A_{n,H}(t)=0$
for $t \geq -(n-1)/H$. Hence, $\tilde{\delta_i}=\min\{\delta_i, -(n-1)/H\}$.

To get volume estimates we just have to evaluate the above integrals. We pick $\delta_2=0$. If $H \geq 0$, $r \leq D$. If $H<0$, by Lemma \ref{lem-focalpoints} $ r \leq -(n-1)/H$. Thus, choose $\delta_1=-(n-1)/H$ when $H<0$ and $\delta_1=D$ otherwise.
\end{proof}

\begin{rmrk}
Since $\cut(r)$ has n-zero measure we can get estimates of the volume of $M^{\delta_2} \setminus M^{\delta_1}$ in a straight forward way. But when calculating estimates of the volume of $\partial M^ \delta$ we can encounter that $\cut(r)$ has $n-1$ nonzero measure or that $\partial M^ \delta$ is not a submanifold. For example, consider a solid hyperboloid in 3-dimensional euclidean space. For an appropriate $\delta$, $\partial M^ \delta$ is exactly two cones that intersect each other at the tip. Hence, the volume of 
$\partial M^ \delta$ is not defined for all $\delta$. 
\end{rmrk}

\begin{proof}[Proof of Theorem \ref{thm-areabounds}]
By Theorem 5.3 in \cite{AK} we know that $\mathcal{L}^1$-almost everywhere 
\be
\vol( \partial M^\delta )=  \frac{d}{dt} \vol(M \setminus M^t)|_{t=\delta},
\ee
and by Theorem \ref{thm-volumebounds} that 
\be 
\vol(M \setminus M^t) \leq \vol(\partial M) \int_{0}^{t}
A_{n,H}(s)ds,
\ee
where $A_{n,H}$ is the continuous function given in (\ref{eq-AHn}). Thus, $\vol( \partial M^\delta) \leq \vol(\partial M)A_{n,H}(\delta)$.
\end{proof}

\begin{rmrk}
Theorem 5.3 in \cite{AK} holds for metric spaces. This exact theorem is the Euclidean Slicing Theorem for the euclidean and manifold setting, Theorem 4.3.2 in \cite{Federer}.
\end{rmrk}

\section{Convergence Theorems}\label{sec-IFC}

In the first subsection we state Wenger Compactness Theorem and Lemma \ref{lem-IFdforsubset} that gives an estimate of the SWIF distance between a manifold and a subset of it. These results are used  in the second subsection to prove Theorem \ref{thm-IFdeltaconvergence} about SWIF convergence of sequences of $\delta$-inner regions; when $\delta=0$ we get Theorem \ref{thm-IFcompact1}. Then we prove Theorem \ref{thm-IFlimits}. At the end of this section we discuss the SWIF convergence, if any, of some sequences of manifolds.
 
\subsection{Sormani-Wenger Intrinsic Flat Convergence}\label{ssect-IFCR}

Federer-Fleming introduced the term "integral current" (lying in Euclidean space) and extended Whitney's notion of flat distance to integral currents. Ambrosio-Kirchheim in \cite{AK} extended Federer-Fleming's integral currents to integral currents lying in arbitrary metric spaces. Later on Sormani-Wenger in \cite{SorWen2} motivated by both Gromov-Hausdorff distance and flat distance defined intrinsic flat distance between $n$-integral current spaces $(X,d_X,T)$. 

In general, $X$ is a countably $\mathcal{H}^n$-rectifiable metric space, $d_X$ is the metric on $X$ and  
$T$ an integral current  in $\intcurr_n(\bar{X})$. See Definition 2.44 in \cite{SorWen2}.  In the setting of manifolds, the $n$-integral current space associated to an oriented manifold $(M^n,g)$ is just $(M,d,T)$, where $d$ is the metric on $M$ induced by $g$ and $T$ is integration over $M$ of top differential forms of $M$, $T(\omega)=\int_{M}\omega$. 

\begin{lem}\label{lem-IFdforsubset}
Let $(M^n,g)$ be an oriented Riemannian manifold and $U$ an open set of $M$.
Then 
\be
d_{\mathcal{F}}((M^n,d,T),(U,d',T')) \leq \vol(M \setminus U),   
\ee
where $d'=d|_U$ and $T'$ is integration over $U$ of top differential forms of $U$. 
\end{lem}

\begin{thm}\label{thm-WengerCompactness}[Wenger, Theorem 1.2 in \cite{Wenger-compactness}]
Given a sequence of complete oriented Riemannian manifolds, $M_j$, of the same dimension with $\vol(M_j)\le V$, $\vol(\partial M_j) \le A$ and $\diam(M_j) \le D$, then a subsequence converges in the SWIF sense to an integral current space. 
\end{thm}

\subsection{SWIF Compactness Theorems}\label{ssec-IFCTheorem}

Given a $\delta$-inner region, $M^{\delta}=r^{-1}((\delta, \infty)) \subset M^n$, we associate to it an $n$-integral current space: $(M^\delta,d_{M^\delta},T^{\delta})$, where $d_{M^\delta}$ is the metric of $M$ restricted to $M^{\delta}$ and $T^{\delta}$ is integration over $M^\delta$ of top differential forms of $M^\delta$, $T^{\delta}(\omega)= \int_{M^\delta}\omega$.

\begin{thm}\label{thm-IFdeltaconvergence}
Let $D, A> 0$ and $H \in \R$. If $(M^n_j,g_j)$ is a sequence of 
n-dimensional complete connected oriented Riemannian manifolds with smooth boundary that satisfy
$(M_j,d_j)$ complete as metric spaces,
\be  
\diam(M_j)\leq D, \,\,\,\ric(M_j \setminus \partial M_j)\geq 0, \,\,\,
\vol(\partial M_j)\leq A\,\,\text{and} \,\, H_{\partial M_j} \leq H,
\ee
then for $\mathcal{L}^1$-a.e. $\delta \geq 0$  there is an $n$-integral current space   
$(W_\delta,d_{W_\delta},T_\delta)$ and a subsequence that depends on $\delta$ such that $(M_{j_k}^\delta, d_{M_{j_k}^\delta},T_{j_k}^\delta) \Fto (W_\delta,d_{W_\delta},T_\delta)$.
\end{thm}

\begin{proof}
The result follows from Theorem \ref{thm-WengerCompactness}. We just need to check that $(M^n_j,g_j)$ satisfies the hypotheses of that theorem. For $\delta=0$, $\diam(M_j) \leq D$, $\vol(\partial M_j) \leq A$ and $\vol(M_j)$ is uniformly bounded by Theorem \ref{thm-volumebounds}. For $\delta>0$, $\diam(M_j^\delta) \leq \diam(M_j) \leq D$, where the metric of $M_j^\delta$ is the restricted metric. By Theorems \ref{thm-volumebounds} and \ref{thm-areabounds} for $\mathcal{L}^1$-a.e. $\delta \geq 0$ $\vol(M^\delta_j) \leq \vol(M_j)$ and $\vol(\partial M^\delta_j)$ are uniformly bounded above. Hence, we can apply Theorem \ref{thm-WengerCompactness}.
\end{proof}

\begin{rmrk}
The proof above consisted on showing that the given sequence satisfy the hypotheses of Theorem \ref{thm-WengerCompactness}. Note that this could not be done if  
$H_{\partial M_j} \to \infty$ since by Theorem \ref{thm-volumebounds} we would get $\lim_{j \to \infty} \vol(M_j)=\infty$. See Example \ref{ex-jfold}.
We also cannot apply Theorem \ref{thm-WengerCompactness} when $\diam(M_j) \to \infty$. See Example \ref{ex-diam}. 
\end{rmrk}

\begin{proof}[Proof of Theorem \ref{thm-IFlimits}]

By the triangle inequality 
\begin{align} \label{ineq-IFlimits}
d_{\mathcal{F}}(W_{\delta_i},W) & \leq 
d_{\mathcal{F}}(W_{\delta_i},M_{j_k}^{\delta_i}) + d_{\mathcal{F}}(M_{j_k}^{\delta_i},M_{j_k})+ d_{\mathcal{F}}(M_{j_k},W).
\end{align}

Now, by Theorem \ref{thm-volumebounds}:
\be 
d_{\mathcal{F}}(M_{j_k}^{\delta_i},M_{j_k}) \leq \vol(M_{j_k} \setminus M_{j_k}^{\delta_i}) \leq V(\delta_i,H,A,n),
\ee
where $V(\delta,H,A,n)$ is a continuous function such that
$\lim_{\delta \to 0}V(\delta,H,A,n)=0$. Then, taking limits in (\ref{ineq-IFlimits}) we get 
\be 
\lim_{i \to \infty} d_{\mathcal{F}}(W_{\delta_i},W)=  \lim_{i \to \infty}\lim_{k \to \infty} d_{\mathcal{F}}(W_{\delta_i},W)=0.
\ee
\end{proof}

\subsection{Examples}\label{ssec-examples}

In this Subsection we present three examples of sequences of Riemannian manifolds.
The first two examples presented are stated for compact manifolds with no boundary but can easily be generalized to manifolds with boundary. Example \ref{ex-IlmanemSeq} defines a sequence that  converges in SWIF sense but not in GH sense.  Example \ref{ex-jfold} shows the necessity of a uniform bound of the mean curvature required in Theorem \ref{thm-IFcompact1}. Example \ref{ex-diam} shows the necessity of a uniform bound of the diameter required in Theorem \ref{thm-IFcompact1} and that equality holds in both volume and area estimates given in Theorem \ref{thm-volumebounds} and Theorem \ref{thm-areabounds}. Moreover, this example shows that the assumption $H<0$ in Theorem \ref{thm-IFcompact2} is needed.

\begin{ex}[Example A.7 in \cite{SorWen2}]\label{ex-IlmanemSeq}
We define below a sequence of manifolds with positive scalar curvature that converges to a round $n$-sphere in the SWIF sense. 

Let $M^n_j$ be diffeormorphic to a $n$-sphere of volume $V$. Suppose that $M_j$ contains a connected open domain $U_j$ isometric to a domain $M_0 \setminus \bigcup_{i=1}^{N_j} B(p_{j,i},R_j)$, where $M_0$ is a round sphere and $B(p_{j,i},R_j)$ are pairwise disjoint balls. Let each connected component of $M_j \setminus U_j$ and each ball $B(p_{j,i},R_j)$ have volume bounded above by $v_j/N_j$ where $v_j \to 0$. Then $M_j$ converges as long as $N_jR^{1/2}_j \to 0$.

The sequence does not converge in GH sense since, for $\varepsilon$ small enough, the number of $\varepsilon$-balls needed to cover $M_j$ goes to infinity as $j$ goes to infinity.
\end{ex}

\begin{ex}[Example 9.1 in \cite{SorAA}]\label{ex-jfold}
Let $M_j$ be the $j$-fold covering space of 
\be 
N_j= S^2 \setminus \left( B(p_+,1/j),B(p_-,1/j)\right),
\ee 
where $(S^2, g_{S^2})$ is the 2-dimensional unit sphere, $S^2$, with the stardand metric. 
The metric of $M_j$ is the lifting of the metric of $N_j$ and $p_+,p_-$ are opposite poles.
Then $\diam(M_j) \leq 4\pi$, $\vol(\partial M_j) \leq 4\pi$ and $H_{\partial M_j} \to \infty$.
No subsequence of $M_j$ converges in SWIF sense, so  $H_{\partial M_j} < H$ for all $j$ is necessary in Theorem \ref{thm-IFcompact1}.
\end{ex}

\begin{ex}\label{ex-diam}
Let $S^k$ be the $k$-dimensional unit sphere and $[0,j] \subset \R$ a closed interval with standard metrics. 
We endow $S^k\times [0,j]$ with the product metric and define 
\be 
M_j:=S^k\times [0,j] / \sim,
\ee 
where we identify antipodal points of $S^k\times \{j\}$.
Thus, $\partial M_j = S^k\times \{0\}$, $H_{\partial M_j}=0$,  $\vol(\partial M_j)=\vol(S^k)$. Note that $\diam(M_j) \to \infty$, $\vol(M_j)=j\vol(S^k) \to \infty$.  This sequence has no SWIF limit. So it proves the necessity of uniformly bounding the diameter of $\{M_j\}$ in Theorem \ref{thm-IFcompact1} and requiring $H <0$ in Theorem \ref{thm-IFcompact2}. 
\end{ex}

\bibliographystyle{plain}
\bibliography{bib2013}

\begin{thebibliography}{10}

\bibitem{Abresch-Gromoll}
Detlef Abresch, Uwe ;~Gromoll.
\newblock On complete manifolds with nonnegative ricci curvature.
\newblock {\em J. Amer. Math. Soc.}, 3:355--374, 1990 no. 2.

\bibitem{AK}
Luigi Ambrosio and Bernd Kirchheim.
\newblock Currents in metric spaces.
\newblock {\em Acta Math.}, 185(1):1--80, 2000.

\bibitem{Anderson-2004}
Michael Anderson, Atsushi Katsuda, Yaroslav Kurylev, Matti Lassas, and Michael
  Taylor.
\newblock Boundary regularity for the {R}icci equation, geometric convergence,
  and {G}el\cprime fand's inverse boundary problem.
\newblock {\em Invent. Math.}, 158(2):261--321, 2004.

\bibitem{B-C}
Richard Bishop and Richard Crittenden.
\newblock {\em Geoemetry of Manifolds}.
\newblock Academic Press, New York, NY, 1964.

\bibitem{Cheeger-Criticalp}
Jeff Cheeger.
\newblock Critical points of distance functions and applications to geometry.
\newblock In {\em Geometric topology: recent developments ({M}ontecatini
  {T}erme, 1990)}, volume 1504 of {\em Lecture Notes in Math.}, pages 1--38.
  Springer, Berlin, 1991.

\bibitem{ChCo-PartI}
Jeff Cheeger and Tobias~H. Colding.
\newblock On the structure of spaces with {R}icci curvature bounded below. {I}.
\newblock {\em J. Differential Geom.}, 46(3):406--480, 1997.

\bibitem{Federer}
Herbert Federer.
\newblock {\em Geometric measure theory}.
\newblock Die Grundlehren der mathematischen Wissenschaften, Band 153.
  Springer-Verlag New York Inc., New York, 1969.

\bibitem{Gromov-metric}
Misha Gromov.
\newblock {\em Metric structures for {R}iemannian and non-{R}iemannian spaces},
  volume 152 of {\em Progress in Mathematics}.
\newblock Birkh\"auser Boston Inc., Boston, MA, 1999.
\newblock Based on the 1981 French original [ MR0682063 (85e:53051)], With
  appendices by M. Katz, P. Pansu and S. Semmes, Translated from the French by
  Sean Michael Bates.

\bibitem{VolComparison}
Ernst Heintze and Hermann Karcher.
\newblock A general comparison theorem with applications to volume estimates
  for submanifolds.
\newblock {\em Annales scientifiques de l'E.N.S}, 11(4):451--470, 1978.

\bibitem{Knox-2012}
Kenneth Knox.
\newblock A compactness theorem for riemannian manifolds with boundary and
  applications.
\newblock {\em arXiv:1211.6210 [math.DG]}, pages 1--17, 2012.

\bibitem{Kodani-1990}
Shigeru Kodani.
\newblock Convergence theorem for {R}iemannian manifolds with boundary.
\newblock {\em Compositio Math.}, 75(2):171--192, 1990.

\bibitem{LiMartin}
Martin Li.
\newblock A sharp comparison theorem for compact manifolds with mean convex
  boundary.
\newblock {\em arXiv:1204.1695v2, [math.DG]}, pages 1--6, 2012.

\bibitem{Li-Nguyen}
YanYan Li and Luc Nguyen.
\newblock A compactness theorem for a fully nonlinear yamabe problem uner a
  lower ricci curvature bound.
\newblock {\em arXiv:1212.0460v1, [math.AP]}, pages 1--34, 2012.

\bibitem{PerSurvey}
Raquel Perales.
\newblock A survey on the convergence of manifolds with boundary.
\newblock {\em arXiv:1310.0850, [math.DG]}, pages 1--9, 2013.

\bibitem{Persor-2013}
Raquel Perales and Christina Sormani.
\newblock Sequences of open riemannian manifolds with boundary.
\newblock {\em Pacific Journal of Mathematics, to appear. arXiv:1301.3961}.

\bibitem{Petersen-text}
Peter Petersen.
\newblock {\em Riemannian geometry}, volume 171 of {\em Graduate Texts in
  Mathematics}.
\newblock Springer, New York, second edition, 2006.

\bibitem{Sakurai}
Yohei Sakurai.
\newblock Rigidity of manifolds with boundary under a lower {R}icci curvature
  bound.
\newblock {\em arXiv:1404.3845v2, [math.DG]}, pages 1--29, 2014.

\bibitem{SorAA}
Christina Sormani.
\newblock Intrinsic flat arzela-ascoli theorems.
\newblock {\em arXiv:1402.6066 [math.MG]}, pages 1--33, 2014.

\bibitem{SorWen1}
Christina Sormani and Stefan Wenger.
\newblock Weak convergence and cancellation, appendix by {R}aanan {S}chul and
  {S}tefan {W}enger.
\newblock {\em Calculus of Variations and Partial Differential Equations},
  38(1-2), 2010.

\bibitem{SorWen2}
Christina Sormani and Stefan Wenger.
\newblock The intrinsic flat distance between riemannian manifolds and other
  integral current spaces.
\newblock {\em Journal of Differential Geometry}, 87:117--199, 2011.

\bibitem{Wenger-compactness}
Stefan Wenger.
\newblock Compactness for manifolds and integral currents with bounded diameter
  and volume.
\newblock {\em Calc. Var. Partial Differential Equations}, 40(3-4):423--448,
  2011.

\bibitem{Wong-2008}
Jeremy Wong.
\newblock An extension procedure for manifolds with boundary.
\newblock {\em Pacific J. Math.}, 235(1):173--199, 2008.

\end{thebibliography}

\end{document}